\documentclass[11pt]{amsart}
\usepackage{enumerate,amsmath,amssymb, mathrsfs}
\usepackage{xcolor}

\usepackage{latexsym, amssymb,enumerate}
\usepackage{appendix}

\newcommand{\be}{\begin{equation}}
\newcommand{\ee}{\end{equation}}
\newcommand{\beq}{\begin{eqnarray}}
\newcommand{\eeq}{\end{eqnarray}}

\def\R{{\mathfrak R}}

\newtheorem{prop}{Proposition}[section]
\newtheorem{theo}[prop]{Theorem}
\newtheorem{lemm}[prop]{Lemma}

\newtheorem{rema}[prop]{Remark}

\def\begeq{\begin{equation}}
\def\endeq{\end{equation}}

\def\R{\mathbb R}

\def\tr{{\rm tr}}

\def\<{\langle}
\def\>{\rangle}
\def\d{\delta}

\def\s{\sigma}

\def \ds{\displaystyle}

\def\S{\mathbb  {S}}

\def\odot{\setbox0=\hbox{$\bigcirc$}\relax \mathbin {\hbox
to0pt{\raise.5pt\hbox to\wd0{\hfil $\wedge$\hfil}\hss}\box0 }}

\numberwithin{equation} {section}

\begin{document}

\title[A note on new weighted geometric inequalities]{A note on new weighted geometric inequalities for hypersurfaces in $\mathbb R^n$}

\author{Jie Wu}
\address{Department of Mathematics, Zhejiang University, Hangzhou 310027, P. R. China}
\email{wujiewj@zju.edu.cn}

\begin{abstract} In this note, we prove a family of sharp weighed inequality which involves weighted $k$-th mean curvature integral and two distinct quermassintegrals for closed hypersurfaces in $\mathbb{R}^n$. This inequality generalizes the corresponding result of Wei and Zhou \cite{WZ} where their proof is based on earlier results of Kwong-Miao \cite{KM1,KM2}. Here we present a proof which does not rely on Kwong-Miao's results.
\end{abstract}




\maketitle


\

\section{Introduction}

The Alexandrov-Fenchel inequalities for quermassintegrals of convex domains in $\mathbb{R}^n$, as a generalization of the isoperimetric inequality, play an important role in classical geometry.
Let $\Omega$ be a smooth bounded domain in $\R^n$ with boundary $\partial\Omega = \Sigma $.
The classical isoperimetric inequality is
\beq\label{eq001}
\frac{|\Sigma|}{\omega_{n-1}}\geq \bigg(\frac{Vol(\Omega)}{\frac{\omega_{n-1}}n}\bigg)^{\frac{n-1}{n}},
\eeq
where $|\Sigma|$ denotes the area of $\Sigma$ and $\omega_{n-1}$ is the area of the unit sphere $\S^{n-1}$. Equality in (\ref{eq001}) holds if and only if $\Omega$ is a geodesic ball.
Let $\kappa=(\kappa_1, \kappa_2, \cdots , \kappa_{n-1})$ be the principal curvatures of $\Sigma $ and for any integer $1\leq k\leq n-1$, define the $k$-th mean curvature $\sigma_k$ and the normalized $k$-th mean curvature $H_k$ of $\Sigma$ by
 \[\s_k=\sum_{1\leq i_1<\cdots<i_{k}\leq{n-1}}\kappa_{i_1}\cdots\kappa_{i_k}\quad  \hbox{and}\quad H_k=\frac{\sigma_k}{C^{k}_{n-1}}.\]
Here we make a convention that $\s_0=H_0=1$ and $\s_k=H_k=0$ for $k\geq n$.
 The celebrated Alexandrov-Fenchel quermassintegral inequalities \cite{AF1,AF2} state that
\begin{equation}\label{eq003}
\frac 1{\omega_{n-1}}\int_{\Sigma} H_kd\mu \ge
\left(\frac 1{\omega_{n-1}}\int_{\Sigma} H_{j} d\mu \right)^{\frac {n-1-k}{n-1-j}}, \quad 0\le j< k\le  n-1,
\end{equation}
for any convex hypersurface  $\Sigma$ and equality holds in (\ref{eq003}) if and only if $\Sigma$ is a geodesic sphere.
By using the inverse curvature flows, Guan and Li \cite{GuanLi} extended the Alexandrov-Fenchel inequalities (\ref{eq003}) to star-shaped and $k$-convex hypersurfaces.
Recently, there is of great interest to establish the weighted geometric inequalities comparing the weighted curvature integrals and the quermassintegrals in $\R^n$.

Though out the paper, let $r$ be the Euclidean distance to any given point $O\in\R^n$.
We say that $\Sigma$ in $\R^n$ is called star-shaped, if its support function $u=\langle X, \nu\rangle$ is positive everywhere on $\Sigma,$ where $X$ is the position function of $\Sigma$ and $\nu$ is the unit outer-normal of $\Sigma$  at $X$.  Here a hypersurface is called $k$-convex if its principal curvatures satisfy $H_i > 0$ for all $1\leq i\leq k$.
In particular, $1$-convex is also called mean convex.

Applying the Reilly's  formula \cite{Reilly}, Kwong and Miao \cite{KM1} first obtained a sharp inequality relating $r^2$-weighted mean curvature and the enclosed volume for a convex hypersurface and then extend it to a star-shaped, mean-convex hypersurface through the inverse mean curvature flow. Precisely, they showed the following inequality
\begin{equation}\label{Miao1}
\int_{\Sigma}r^2 H_1\geq n Vol(\Omega),
\end{equation}
if $\Sigma$ is star-shaped and mean convex. Equality holds if and only if $\Sigma$ is a geodesic sphere.

By using the  inequality (\ref{Miao1}) and the inverse mean curvature flow, Gir\~{a}o and Rodrigues \cite{GR} proved for a star-shaped and mean convex hypersurface that
\begin{equation}\label{Girao}
\int_{\Sigma}r^2 H_1\geq \omega_{n-1}\bigg(\frac{|\Sigma|}{\omega_{n-1}}\bigg)^{\frac{n}{n-1}}.
\end{equation}
Inequality (\ref{Girao}) can be seen as an improvement of (\ref{Miao1}) by the classic isoperimetric equality (\ref{eq001}).

For higher order mean curvature integrals, Kwong and Miao \cite{KM2} also established a family of weighted inequalities as following:
\begin{theo}[\cite{KM2}]
Let $\Sigma$ be a smooth, closed, star-shaped and k-convex hypersurface in $\R^n(n\geq 3)$.
Then for all $k=2,\cdots, n-1$, there holds
\begin{equation}\label{Miao2}
       \int_{\Sigma}r^2 H_k\geq \int_{\Sigma} H_{k-2}.
     \end{equation}
     Moreover, the equality holds if and only if $\Sigma$ is a round sphere.
\end{theo}
The method to prove (\ref{Miao2}) is the use of a generalized Hsiung-Minkowski formula\cite{Kwong}.

More recently, by using (\ref{Miao2}) and the inverse curvature flows, Wei and Zhou \cite{WZ} established a sharp weighted geometric inequality with the right hand-side of inequality (\ref{Miao2}) replaced by $\int_{\Sigma}H_{k-1}$. These inequalities can be seen as the analogs of (\ref{Girao}) for general $k$.

\begin{theo}[\cite{WZ}]
Let $\Sigma$ be a smooth, closed, star-shaped and strictly k-convex hypersurface in $R^n(n\geq 3)$.
Then for all $k = 2,\cdots, n-1 $, there holds
\begin{equation}\label{Wei1}
       \int_{\Sigma}r^2 H_k\geq \omega_{n-1}\bigg(\frac{\int _{\Sigma} H_{k-1}}{\omega_{n-1}}\bigg)^{\frac{n-k+1}{n-k}}.
\end{equation}
Moreover, the equality holds if and only if $\Sigma$ is a round sphere.
\end{theo}
Inequality(\ref{Wei1}) also improves (\ref{Miao2}) since (\ref{Wei1}) implies (\ref{Miao2}) by the Alexandrov-Fenchel inequality for star-shaped and $k$-convex hypersurfaces. By the same reason, they acturally obtained a complete family of sharp weighted inequalities following from (\ref{Wei1}) with the right hand-side of (\ref{Miao2}) replaced by suitable power of $\int_{\Sigma}H_{l}\;(l\leq k-1)$  as stated in their paper (see Theorem 1.2 in \cite{WZ}).

However, we notice that the two proofs of inequalities (\ref{Wei1})  and (\ref{Girao}) both relies on (\ref{Miao2}) and (\ref{Miao1}) respectively. It is natural to ask  if one can give a direct proof of (\ref{Wei1})  and (\ref{Girao}) without using Kwong-Miao's results (\ref{Miao2}) and (\ref{Miao1}). We give an affirmative answer in this short note. Motivated by \cite{GuanLi}, we observe that it is more convenient to use the following normalized inverse curvature flow
\begin{equation}\label{flow1}
\frac{\partial}{\partial t} X=\bigg(\frac {H_{k-1}}{H_k}-u\bigg)\nu,
\end{equation}
rather than the un-normalized one. It's well-known that flow (\ref{flow1}) is equivalent to
the inverse curvature flow
\begin{equation}\label{flow2}
\frac{\partial}{\partial t} X=\frac {H_{k-1}}{H_k}\nu,
\end{equation}
introduced by Gerhardt \cite{Gerhardt} and Urbas \cite{Urbas} by rescaling.

Furthermore, by using (\ref{flow1}), we are able to establish a new kind of weighed inequality which involves weighted $k$-th mean curvature integral as well as two distinct quermassintegrals. This inequality improves (\ref{Wei1}) by the Alexandrov-Fenchel inequality (\ref{eq003})  for star-shaped and $k$-convex hypersurfaces.

\begin{theo}
Let $\Sigma$ be a smooth, closed, star-shaped and $k$-convex hypersurface in $\mathbb{R}^n(n\geq 3)$.
Then for each $k= 1,\cdots, n-1,$ there holds
\begin{equation}\label{W1}
       \int_{\Sigma}r^2 H_k d\mu +\frac{2(k-1)}{n-k+1}\int_{\Sigma} H_{k-2} d\mu \geq \frac{n+k-1}{n-k+1} \omega_{n-1}\bigg(\frac{\int _{\Sigma} H_{k-1}d\mu}{\omega_{n-1}}\bigg)^{\frac{n-k+1}{n-k}}.
\end{equation}
Moreover, equality holds in (\ref{W1}) if and only if $\Sigma$ is a round sphere.
\end{theo}
\begin{rema}
\item $(1)$ When $k=1$, we make a convention that $\int_{\Sigma} H_{-1}d\mu= nVol(\Omega)$. Thus (\ref{W1}) reduces to  Gir\~{a}o and Rodrigues's result (\ref{Girao}).
\item $(2)$ We remark that quite recently Kwong and Wei \cite{KW} obtained a slightly stronger result than (\ref{W1}) with a proof also using Kwong-Miao's results (\ref{Miao1}) and (\ref{Miao2}), while our proof presented here does not rely on them. As in turn, our proof provides another approach to the proof of Kwong-Miao's results (\ref{Miao1}) and (\ref{Miao2}).
\end{rema}

\section{Preliminaries}

Let $(\Sigma^{n-1},g)$ be a hypersurface in $\R^{n}$ with $g$ the induced metric and $h$ its second fundamental form. Let $\{x^1, \cdots x^{n-1}\}$ be a local coordinate system of $\Sigma$  and denote by $g_{ij}=g(\partial_i, \partial_j)$ and $h_{ij}=h(\partial_i, \partial_j)$ . The pricipal curvature $\kappa=(\kappa_1, \cdots, \kappa_{n-1})$ of $\Sigma$ are the eigenvalues of shape opertor $h_i^j$, where $h_i^j=g^{jk}h_{ik}$ and $(g^{ij})$ is the inverse matrix of $(g_{ij})$.  The $k$-th Newton operator of $h_i^j$ is defined by
\begin{equation}\label{Newtondef}
(T_k)^{i}_{j}:=\frac{\partial \s_{k+1}}{\partial h^{j}_{i}}.
\end{equation}

The following basic formulas about $\s_k$ and $T$ are well-known, see e.g. \cite{GuanLi, Reilly}.
\begin{eqnarray}\label{sigmak}
\s_k(h)& =&\ds\frac{1}{k!}\d^{i_1\cdots i_k}
_{j_1\cdots j_k}h_{i_1}^{j_1}\cdots
{h_{i_k}^ {j_k}}=\frac{1}{k} \tr(T_{k-1}(h)h),
\\
(T_k)^{i}_j(h) & =& \ds \frac{1}{k!}\d^{ii_1\cdots i_{k}}
_{j j_1\cdots j_{k}}h_{i_1}^{j_1}\cdots
{h_{i_{k}}^{j_{k}}}\label{Tk},
\end{eqnarray}
where $\d^{i_1i_2\cdots i_{k}}
_{j_1j_2\cdots j_{k}}$ is the generalized Kronecker delta given by
\begin{equation}\label{generaldelta}
 \d^{i_1i_2 \cdots i_k}_{j_1j_2 \cdots j_k}=\det\left(
\begin{array}{cccc}
\d^{i_1}_{j_1} & \d^{i_2}_{j_1} &\cdots &  \d^{i_k}_{j_1}\\
\d^{i_1}_{j_2} & \d^{i_2}_{j_2} &\cdots &  \d^{i_k}_{j_2}\\
\vdots & \vdots & \vdots & \vdots \\
\d^{i_1}_{j_k} & \d^{i_2}_{j_k} &\cdots &  \d^{i_k}_{j_k}
\end{array}
\right).
\end{equation}

  We recall the following lemma for smooth hypersurfaces in $\R^n$, see for example \cite{Guan, GuanLi2}.
  \begin{lemm}
  Let $(\Sigma^{n-1},g)$ be a smooth closed hypersurface in $\R^n$. Then there holds
  \begin{equation}\label{nabla}
  \nabla_i(T^{k-1}_{ij}\nabla_j r^2)=2((n-k)\sigma_{k-1}-k\sigma_{k}u)=2kC_{n-1}^k(H_{k-1}-uH_k),
  \end{equation}
  where $\nabla$ is the covariant derivative with respect to the metric $g$.
  \end{lemm}
Integrating above equality, we get the well-known Minkowski formulas in $\R^n$:
\begin{equation}\label{Minkowski}
\int_{\Sigma} u H_{k}d\mu= \int_{\Sigma} H_{k-1} d\mu.
\end{equation}
We also need the following integral equalities which play an important role in our proof.

\begin{prop}\label{prop1}
Let $\Sigma^{n-1}$ be a closed hypersurface in the Euclidean space $\R^n$. For $k= 1,\cdots,n-1 $, we have
\begin{equation}\label{eq1}
\int_{\Sigma}r^2 u H_kd\mu= \int_{\Sigma}r^2H_{k-1}d\mu+\frac{1}{2kC_{n-1}^k}\int_{\Sigma}(T_{k-1})^{ij}\nabla_i (r^2) \nabla_j (r^2)d\mu,
\end{equation}
and
\begin{equation}\label{eq2}
\int_{\Sigma}{u}^2H_k d\mu=\int_{\Sigma}{u}H_{k-1} d\mu+\frac{1}{4kC_{n-1}^k} \int_{\Sigma} (T_{k-1})^{ij}h_i^l\nabla_j (r^2) \nabla_l (r^2)d\mu.
\end{equation}
\end{prop}
\begin{proof}
The proof is quite simple. First multiplying (\ref{nabla}) by the function $r^2$ and integrating by parts,
(\ref{eq1}) follows immediately.  Similarly, multiplying (\ref{nabla}) by the support function $u$ and integrating by parts, we obtain
(\ref{eq2}) by  the fact (see \cite[Lemma 2.6]{GuanLi2} for instance)
$$\nabla_i u= h_i^l \nabla_l\bigg(\frac 12 r^2\bigg).$$
\end{proof}

We recall the following variational formula for the quermassintegral (see e.g. \cite{BC,GuanLi}):
\begin{equation}\label{variation}
\frac{d}{dt} \int_{\Sigma_t}H_{k-1} d\mu_t=(n-k)\int_{\Sigma_t} H_{k}F d\mu_t,\qquad k=1,\cdots n-1
\end{equation}
along the general normal variation
$$\frac{\partial}{\partial t}X=F\nu.$$

Finally, we need a general variation equation for the $r^2$-weighted curvature integral $\int_{\Sigma_t} r^2 H_k$ as following:
\begin{lemm}
Let $\Sigma_t$ be a smooth family of closed hypersurfaces in the Euclidean space $\mathbb{R}^n$ satisfying a general flow
\begin{equation}\label{general flow}
\frac {\partial}{\partial t} X=F\nu,
\end{equation}
where $\nu$ is the outward unit normal of $\Sigma_t$ and $F$ is a smooth function on $\Sigma_t$. For all $k=1,\cdots, n-1$, we have
\begin{equation}\label{Vpk}
\frac{\partial}{\partial t}\int_{\Sigma}r^2 H_k d\mu_t=\int_{\Sigma}\bigg((n-1-k)\,r^2 H_{k+1}+2(k+1){u} H_k -2kH_{k-1}\bigg)Fd\mu_t.
\end{equation}
\end{lemm}
\begin{proof}
The variational formula (\ref{Vpk}) is  established in \cite{WZ}. For the convenience of readers, we include the proof here.
First note that
\begin{equation*}
\frac{\partial}{\partial t} r^2 = \langle\bar\nabla r^2,\frac{\partial}{\partial t} X \rangle
 = 2\langle r\partial r, F\nu\rangle = 2 Fu,
\end{equation*}
where $\bar\nabla$ is the covariant derivative with respct to the Euclidean metric.

In view of (\ref{nabla}) and the divergence-free property of $T_{k-1}$, we have along the flow (\ref{general flow}) that
\begin{align}
\frac{d}{d t}\int_{\Sigma_t} r^2 \sigma_k d\mu_t&=\int_{\Sigma_t}\bigg(\frac{\partial}{\partial t} r^2\bigg)\sigma_k d\mu_t+\int_{\Sigma_t}r^2\frac{\partial\sigma_k}{\partial t} d\mu_t+\int_{\Sigma_t}r^2\sigma_k (F\sigma_1)d\mu_t\nonumber\\
&=\int_{\Sigma_t}2 Fu\sigma_k\!+\!\int_{\Sigma_t}\!r^2\bigg(\!-\!T_{k-1}^{ij}\nabla_i\nabla_j F\!-\!F\big(\sigma_1\sigma_k\!-\!(k\!+\!1)\sigma_{k+1}\big)\bigg)d\mu_t
\!+\!\int_{\Sigma_t}\!r^2 F\sigma_k\sigma_1 d\mu_t\nonumber\\
&=\int_{\Sigma}\bigg(2 u\sigma_k- \nabla_i(T_{k-1}^{ij} \nabla_j r^2) +(k+1)r^2 \sigma_{k+1}\bigg) F d\mu_t\nonumber\\
&=\int_{\Sigma}\big(2(k+1)u\sigma_k + (k+1) r^2\sigma_{k+1}- 2(n-k)\sigma_{k-1}\big) F d\mu_t.\nonumber
\end{align}
Thus
(\ref{Vpk}) follows from the normalization
$H_k=\frac{\s_k}{C_{n-1}^k}.$
\end{proof}

\section{Proof of the main results}
In this section, we give a direct proof of (\ref{W1}) without using Kwong-Miao's results (\ref{Miao1}) and (\ref{Miao2}). The idea is to choose the normalized flow (\ref{flow1}) of the inverse curvature flow rather than the un-normalized one.
\medskip

{\it The proof of Theorem 1.3}:

For all $k=1,\cdots, n-1$, first by choosing
    $$F =\frac {H_{k-1}}{H_k}-u,$$
in (\ref{general flow}), we derive from (\ref{variation}) and (\ref{Vpk}) that
\begin{align*}
&\frac{d}{dt} \bigg(\int_{\Sigma_t} r^2 H_k d\mu_t + \frac{2(k-1)}{n+1-k}\int_{\Sigma} H_{k-2} d\mu_t\bigg)\\
=& \int_{\Sigma_t}\bigg((n-1-k) r^2 H_{k+1}+ 2(k+1)u H_k-2kH_{k-1}+2(k-1)H_{k-1}\bigg)\bigg(\frac {H_{k-1}}{H_k}-u\bigg)d\mu_t\\
=& \int_{\Sigma_t}\bigg((n-1-k) r^2 H_{k+1}+ 2k u H_k+2(uH_k-H_{k-1})\bigg)\bigg(\frac {H_{k-1}}{H_k}-u\bigg)d\mu_t\\
=&\, (n-1-k)\int_{\Sigma_t}r^2 \bigg( \frac{H_{k+1}H_{k-1}}{H_k}-u H_{k+1}\bigg)d\mu_t+2k\int_{\Sigma_t}(uH_{k-1}-u^2 H_k)d\mu_t\\
&-2\int_{\Sigma_t} H_{k}\bigg(\frac{H_{k-1}}{H_k}-u\bigg)^2 d\mu_t\\
\leq&(n-1-k)\int_{\Sigma_t}r^2 \bigg( H_k-u H_{k+1}\bigg)d\mu_t -\frac{1}{2C_{n-1}^k} \int_{\Sigma_t} (T_{k-1})^{ij}h_i^l\nabla_j (r^2) \nabla_l (r^2)d\mu_t,
\end{align*}
where in the last inequality we used (\ref{eq2}) and the Newton-MacLaurin inequality
$$H_{k+1}H_{k-1}\leq H_k^2.$$
To deal with the second term in the above inequality, we apply the  well-known iteration of the Newton operator
$$T_k=\sigma_k I-T_{k-1}h,$$
to compute that
\begin{align*}
&\frac{d}{dt} \bigg(\int_{\Sigma_t} r^2 H_k d\mu_t + \frac{2(k-1)}{n+1-k}\int_{\Sigma} H_{k-2} d\mu_t\bigg)\\
\leq&~(n-1-k)\int_{\Sigma_t}r^2 \bigg( H_k-u H_{k+1}\bigg)d\mu_t -\frac{1}{2kC_{n-1}^k} \int_{\Sigma_t} \bigg(-(T_{k})^{jl}+\sigma_k g^{jl}\bigg)\nabla_j (r^2) \nabla_l (r^2)d\mu_t\\
=&~(n-1-k)\int_{\Sigma_t}r^2 \bigg( H_k-u H_{k+1}\bigg)d\mu_t+
\frac{2(k+1)C_{n-1}^{k+1}}{2C_{n-1}^k} \int_{\Sigma_t} r^2(uH_{k+1}-H_k)d\mu_t\\
&-\frac{1}{2C_{n-1}^k} \int_{\Sigma_t} \sigma_k g^{jl} \nabla_j (r^2) \nabla_l (r^2)d\mu_t\\
=&-\frac{1}{2} \int_{\Sigma_t} H_k |\nabla (r^2)|^2d\mu_t,
\end{align*}
where we plugged into (\ref{eq1}) in the second equality.
Since $H_k\geq 0$,  we derive
\begin{equation}\label{eq5}
 \frac{d}{dt} \bigg(\int_{\Sigma_t} r^2 H_k d\mu_t + \frac{2(k-1)}{n+1-k}\int_{\Sigma} H_{k-2} d\mu_t\bigg)\leq 0.
\end{equation}
On the other hand, the $k$th quermassintegral is preserved by the Minkowski formula (\ref{Minkowski}) that
\begin{equation}\label{eq6}
\frac{d}{dt} \int_{\Sigma_t} H_{k-1} d\mu_t=(n-1-k)\int_{\Sigma_t} H_{k}\big(\frac {H_{k-1}}{H_k}-u\big)d\mu_t=0.
\end{equation}

Combining the (\ref{eq5}) and (\ref{eq6}) together, we have for each $k=1,\cdots, n-1$,  the function
$$Q_k(t)= \bigg(\frac{\int _{\Sigma} H_{k-1}}{\omega_{n-1}}\bigg)^{-\frac{n-k+1}{n-k}}\bigg(\int_{\Sigma_t} r^2 H_k d\mu_t + \frac{2(k-1)}{n+1-k}\int_{\Sigma} H_{k-2} d\mu_t\bigg)\\$$
is monotone decreasing along the normalized flow (\ref{flow1}).

If the initial hypersurface $\Sigma$ is star-shaped and $k$-convex, Gerhardt \cite{Gerhardt} an Urbas\cite{Urbas} proved that star-shapeness and the $k$-convexity are preserved and the solution $\Sigma_t$ expands to infinity and properly rescaled solution converges to a round sphere $S_{r_{\infty}}$ as $t\rightarrow \infty$. Note that the quantity $Q_k(t)$ is a scaling invariant, we have
$$Q_k(0)\geq \lim_{t\rightarrow\infty} Q_k(t)=\frac{n+k-1}{n-k+1}\omega_{n-1}.$$

Therefore we obtain the inequality (\ref{Girao}) and (\ref{Wei1}) for a smooth, star-shaped and  $(k+1)$-convex hypersurface in $\R^n.$ Moreover if equality holds, then
$\frac{d}{dt} Q_k(t)=0$ holds on the solution of $\Sigma_t$ of the flow (\ref{flow1}) for all time t which implies that the initial hypersurface $\Sigma$ is a coordinate sphere. This completes the proof.
\qed

\begin{rema}
  From the above proof, we actually have the following integral formula along the flow
 $$
\frac{\partial}{\partial t} X=\bigg(\frac {H_{k-1}}{H_k}-u\bigg)\nu,
$$
that
  \begin{align*}
&\frac{d}{dt} \bigg(\int_{\Sigma_t} r^2 H_k d\mu_t + \frac{2(k-1)}{n+1-k}\int_{\Sigma} H_{k-2} d\mu_t\bigg)\\
=&\, (n-1-k)\int_{\Sigma_t}r^2 \bigg( \frac{H_{k+1}H_{k-1}}{H_k}-H_k\bigg)d\mu_t-2\int_{\Sigma_t} H_{k}\bigg(\frac{H_{k-1}}{H_k}-u\bigg)^2 d\mu_t-\frac{1}{2} \int_{\Sigma_t} H_k |\nabla (r^2)|^2d\mu_t,
\end{align*}
which may have the independent interest.
\end{rema}

\end{document}